\documentclass
{amsart}

\usepackage{amssymb,amsmath,tikz,stmaryrd}
\usetikzlibrary{positioning}
\usetikzlibrary{calc}
\usetikzlibrary{decorations.markings}
\usepackage[bookmarks=true,%
    colorlinks=true,%
    linkcolor=blue,%
    citecolor=blue,%
    filecolor=blue,%
    menucolor=blue,%
    urlcolor=blue,%
    breaklinks=true]{hyperref}

\newcommand{\Z}{\mathbb{Z}}
\newcommand{\C}{\mathbb{C}}
\newcommand{\R}{\mathbb{R}}

\newcommand{\mym}[1]{\operatorname{\tau}_{#1}}
\newcommand{\myc}[1]{\operatorname{c}_{#1}}
\newcommand{\myr}[1]{\operatorname{r}_{#1}}
\newcommand{\mysm}[1]{\operatorname{SM}(#1)}

\newcommand{\mysms}[2]{\operatorname{SM}_{#1}(#2)}
\newcommand{\mytr}[1]{{#1}^{\mathrm{T}}}
\newcommand{\myem}{[]}
\newtheorem{theorem}{Theorem}
\newtheorem{lemma}{Lemma}

\theoremstyle{example}

\newtheorem{remark}{Remark}

\tikzstyle arrowstyle=[scale=1]
\tikzstyle directed=[postaction={decorate,decoration={markings,
    mark=at position .85 with {\arrow[arrowstyle]{stealth}}}}]

\title{On symmetric matrices associated with oriented link diagrams}
\author{Rinat Kashaev}
\address{Section de math\'ematiques, Universit\'e de Gen\`eve,
2-4 rue du Li\`evre, 1211 Gen\`eve 4, Suisse\\}
\email{rinat.kashaev@unige.ch}
\date{January 18, 2018}

\thanks{Supported in part by the Swiss National Science Foundation}
\begin{document}
\begin{abstract} 
Let  $D$ be an oriented link diagram with the set of regions $\myr{D}$.  We define a symmetric map (or matrix) $\mym{D}\colon \myr{D}\times  \myr{D} \to \Z[x]$ that gives rise to an invariant of oriented links, based on a slightly modified $S$-equivalence of Trotter and Murasugi in the space of symmetric matrices. In particular, for real $x$, the negative signature of $\mym{D}$ corrected by the writhe is conjecturally twice the Tristram--Levine signature function, where  $2x=\sqrt{t}+\frac1{\sqrt{t}}$ with $t$ being the indeterminate of the Alexander polynomial. 
\end{abstract}
\maketitle
\section{Introduction}
In this paper, we describe a result obtained in an attempt to understand the metaplectic invariants of Goldschmidt--Jones~\cite{MR1012438}, generalizing the cyclotomic invariants of Kobayashi--Murakami--Murakami~\cite{MR974081}, from the standpoint of (quantum) discrete quantum field theoretical (or complexified statistical mechanics) models on link diagrams. The Goldschmidt--Jones construction is based on the unitarity of the Burau representation discovered by Squier~\cite{MR727232}. Eventually, one can indeed construct an IRF model~\cite{MR690578,MR990215} specified by a choice of a Pontryagin self-dual locally compact abelian group, together with a gaussian exponential and a group endomorphism symmetric with respect to the non-degenerate bi-character induced from the gaussian exponential. Leaving this part for a separate paper, here, we consider the classical (non-quantum) level, where all these models are controlled by a universal symmetric matrix $\mym{D}\colon \myr{D}\times  \myr{D} \to \Z[x]$ indexed by the set $\myr{D}$ of regions of an oriented link diagram $D$ and with coefficients taking their values in the space $\Z[x]$ of polynomials  in one indeterminate with integer coefficients. This matrix plays the role of an Alexander matrix in the description of the Alexander module of knot complements, and the indeterminate $x$ is expected to be related to the indeterminate $t$ of the Alexander module through the formula $2x=\sqrt{t}+\frac1{\sqrt{t}}$. Given the fact that our matrix is 
symmetric, one could try to use  the $S$-equivalence of Trotter--Murasugi in combination with congruences, in order to extract a link invariant by using the signature of the matrix. 
The result of this construction is the content of Theorem~\ref{thm-sr-equiv}. 

The paper is organized as follows. In the next Section~\ref{sec1}, we explain our setup, 
together with the notation and conventions used. In Section~\ref{sec2}, we define the matrix $\mym{D}$, formulate the main Theorem~\ref{thm-sr-equiv}, and present its proof. In the last Section~\ref{sec3}, we describe the results for few simplest examples.
\subsection*{Acknowledgements} I would like to thank Louis-Hadrien Robert for interesting discussions. The work is supported in part by the Swiss National Science Foundation.

\section{Notation, terminology, definitions, and conventions} \label{sec1}
\subsection{Cardinality of sets} For any finite set $X$, we denote by $|X|$ the number of elements in it.
\subsection{Sets of maps} For two sets $A$ and $B$, $A^B$ denotes the set of all maps from $B$ to $A$.

 \subsection{Integers as the finite ordinals} By abuse of notation, we will find it convenient to identify each non-negative integer $n\in\mathbb{Z}_{\ge0}$ with the corresponding von Neumann ordinal number: $0=\{\}$,  $1=\{0\}$, $2=\{0,1\}$, and so on. Thus, the writing $i\in n$ is equivalent to the writing $i\in \{0,1,\dots, n-1\}=\mathbb{Z}_{< n}\cap \mathbb{Z}_{\ge 0}$. That means that we have only set-theoretical isomorphisms like $2^3\simeq 8$. Nonetheless, in the most of the cases, we can still assume $2^3=8$, provided no contextual confusion occurs. 
\subsection{Indexation by lower indices} For positive integers $m,n,\dots$, a set $X$, and a map $f\colon m\times n\times \dots \to X$, we mostly write $f_{i,j,\dots}$ or $(f)_{i,j,\dots}$ instead of $f(i,j,\dots)$ where $(i,j,\dots)\in m\times n\times\dots$. 
\subsection{Charateristic functions} For two sets $A$ and $B$, $\chi_A\in 2^B$ stands for the characteristic function of the subset $A\cap B\subset B$, i.e.
\begin{equation}
\chi_A(x)=\left\{
\begin{array}{cl}
 1&\text{if } x\in A\cap B,\\
 0&\text{if } x\in B\setminus A.
\end{array}
\right.
\end{equation}
For example, the function $\chi_{\{\sqrt{2}\}}$ is defined  on any set, containing or not $\sqrt{2}$. It is also convenient to write $\delta_{x,y}$ instead of $\chi_{\{x\}}(y)$. In the case, where $A=B$, we also informally write $1$ instead of $\chi_A$.
\subsection{Conventions on vectors and matrices} Let $R$ be a  commutative ring (with unit) and $X$, a finite set.  A \emph{\color{blue} vector} over $R$ indexed by $X$ is, by definition, an element of the additive group $R^{X}$. A vector is called  \emph{\color{blue}explicit} if it is indexed by $X\subset n$ for some positive integer $n$.

If the indexing set is a cartesian product $X\times Y$, then the vector is called \emph{\color{blue} matrix}. A matrix is called  \emph{\color{blue} explicit} if it is indexed by  $X\times Y$ where  $X\subset m$ and $Y\subset n$ for some positive integers $m$ and $n$.

     The \emph{\color{blue} transpose} of a matrix $f\in R^{X\times Y}$ is the matrix $\mytr{f}\in R^{Y\times X}$ defined by $\mytr{f}(y,x)=f(x,y)$ for all  $(x,y)\in X\times Y$.  
     
     By convention, any vector $v\in R^X$ is also interpreted as a special case of a matrix in the form of a \emph{\color{blue} column vector} $v\in R^{X\times 1}$, and its transpose is called  \emph{\color{blue} row vector}. 
     
     A matrix $f\in R^{X\times X}$ is called \emph{\color{blue}symmetric} if $\mytr{f}=f$. 
     
     The \emph{\color{blue}product} of two matrices $f\in R^{X\times Y}$ and $g\in R^{Y\times U}$ is the matrix $fg\in R^{X\times U}$ defined  by $(fg)(x,u)=\sum_{y\in Y}f(x,y) g(y,u)$. By abuse of notation, for any finite set $X$, we let $1$ denote the characteristic function of the diagonal $\Delta(X)\subset X\times X$, and call it \emph{\color{blue} identity matrix}.  
     
     Vectors $v,w\in R^X$ are called \emph{\color{blue} orthogonal} if $\mytr{v}w=0$.
     
     Two matrices $f\in R^{X\times Y}$ and $g\in R^{U\times V}$ are called \emph{\color{blue}equivalent} if there exist  bijections $p\colon X\to U$ and $q\colon Y\to V$ such that $f(x,y)=g(p(x),q(y))$ for all $(x,y)\in X\times Y$.  For any matrix $f$, there exists a  unique pair $(m,n)$ of integers such that $f$ is equivalent to at least one explicit matrix indexed by $m\times n$. The latter is called the \emph{\color{blue} size} of $f$.  In the case of vectors, the sizes $n\times 1$ and $1\times n$ are called \emph{\color{blue} length} $n$. 
     \subsection{The direct sum of vectors and matrices}
     Given two explicit vectors $u\in R^m$ and $v\in R^n$. Their direct sum is the vector $u\oplus v\in R^{m+n}$ defined by
     \begin{equation}
(u\oplus v)_{i}=
\left\{
\begin{array}{cl}
  u_{i}  & \text{if}\quad i\in m,  \\
  v_{i-m}  &   \text{if}\quad i-m\in n.
\end{array}
\right.
\end{equation}
Likewise for matrices. Given two explicit matrices $A\in R^{k\times l}$ and $B\in R^{m\times n}$. Their direct sum is the  matrix $A\boxplus B\in R^{(k+m)\times (l+n)}$ defined by
     \begin{equation}
(A\boxplus B)_{i,j}=
\left\{
\begin{array}{cl}
  A_{i,j}  & \text{if}\quad (i,j)\in k\times l,  \\
  B_{i-k,j-l}  &   \text{if}\quad (i-k,j-l)\in m\times n,  \\
 0 &   \text{otherwise}.
\end{array}
\right.
\end{equation}
Notice that the direct sum of two vectors of lengths $m$ and $n$ is a vector of length $m+n$, while their direct sum as column vectors is a matrix of size $(m+n)\times 2$ which is not a column vector.
     \subsection{Pushing forward vectors and matrices}
For finite sets $X,Y$,  let $p\colon X\to Y$  be a map. As any maps, vectors can be pulled back by $p$:
\begin{equation}
p^*\colon R^Y\to R^X,\quad p^*(v)=v\circ p.
\end{equation}
 Due to finiteness of underlying sets, vectors can also be pushed forward. The \emph{\color{blue} push-forward} of a vector $v\in R^{X}$ by $p$ is the vector $p_*(v)\in R^{Y}$ defined by the formula
 \begin{equation}
p_*(v)(y):=\sum_{x\in p^{-1}(y)}v(x),\quad \forall y\in Y.
\end{equation}
One has the following relation between the pull-back $p^*$ and the push-forward $p_*$.
\begin{lemma}
  Let $p\colon X\to Y$ be a map of finite sets. Then, 
  \begin{equation}
\mytr{p_*(v)}u=\mytr{v}p^*(u),\quad \forall (v,u)\in R^X\times R^Y.
\end{equation}
\end{lemma}
\begin{proof}
 For any vectors $u\in R^Y$ and $v\in R^X$,  we have
\begin{multline}
\mytr{p_*(v)}u=\sum_{x\in X,y\in Y} \chi_{ p^{-1}(y)}(x)v(x)u(y)=\sum_{x\in X,y\in Y} \chi_{ \{y\}}(p(x))v(x)u(y)\\
=\sum_{x\in X,y\in Y}\delta_{y,p(x)}v(x)w(y)=\sum_{x\in X} v(x)u(p(x))=\sum_{x\in X} v(x)p^*(u)(x)=\mytr{v}p^*(u).
\end{multline}
\end{proof}
The \emph{\color{blue} push-forward} of a matrix $f\in R^{X\times Y}$ by a pair of maps  $p\colon X\to U$ and $q\colon Y\to V$
is the push-forward $(p\times q)_*(f)\in R^{U\times V}$. If $f$ is symmetric, and $p=q$, then $(p\times p)_*(f)$ is also symmetric.  It is denoted $p_*(f)$ and simply called  the push-forward of $f$ by $p$.

For a map between finite sets $p\colon X\to Y$, a vector $w\in R^X$ is called \emph{\color{blue} $p$-compatible} if $p^*(p_*(w))=w$. We remark that any vector $w\in R^X$ is $p$-compatible if $p$ is injective.

For any $p$-compatible $w$ and any matrix $A$ indexed by $Z\times X$, one has
\begin{equation}\label{proppf}
Aw=(\operatorname{id}_Z\times p)_*(A)p_*(w).
\end{equation}
 Indeed, for any $z\in Z$, we have 
\begin{multline}
(Aw)(z)=\mytr{\left(\chi_{\{z\}}\right)}Aw=\mytr{\left(\chi_{\{z\}}\right)}Ap^*(p_*(w))=\mytr{p_*(w)}p_*(\mytr{A}\chi_{\{z\}})\\
=\mytr{p_*(w)}(p\times \operatorname{id}_Z)_*(\mytr{A})\chi_{\{z\}}=\mytr{p_*(w)}\mytr{\left((\operatorname{id}_Z\times p)_*(A)\right)}\chi_{\{z\}}\\
=
\mytr{\left(\chi_{\{z\}}\right)}(\operatorname{id}_Z\times p)_*(A)p_*(w)=\left((\operatorname{id}_Z\times p)_*(A)p_*(w)\right)(z).
\end{multline}

\begin{lemma}\label{3-equiv-cond}
 Let $p\colon X\to Y$ be a map of finite sets and $u$, a vector indexed by $Y$. Then, the following two conditions are equivalent, and they imply that $p^*(u)$ is $p$-compatible.
\begin{itemize}
 \item[(i)] One has the equality $u=p_*(p^*(u))$. 
  \item[(ii)]   For any $y\in Y$, one has $ u(y)(|p^{-1}(y)|-1)=0$.
\end{itemize}
Moreover, if $p$ is surjective, then $p$-compatibility of $p^*(u)$ implies above conditions.
\end{lemma}
\begin{proof}
  ((i)$\Leftrightarrow $(ii))
For any vector $u$ indexed by $Y$, and any $y\in Y$, we have
\begin{equation}
p_*(p^*(u))(y)=\sum_{x\in X}\chi_{p^{-1}(y)}(x)u(p(x))=u(y)|p^{-1}(y)|
\end{equation}
so that 
\begin{equation}
p_*(p^*(u))=u\Leftrightarrow u(y)(|p^{-1}(y)|-1)=0,\quad \forall y\in Y.
\end{equation}

Denoting $w:=p^*(u)$, we have $ p^*(p_*(w))=p^*(p_*(p^*(u)))=p^*(u)=w$, i.e. $w$ is $p$-compatible.

Now, assuming that $p$ is surjective and $w:=p^*(u)$ is $p$-compatible, define $v:=p_*(w)$. Then, $p^*(v)=p^*(p_*(w))=w=p^*(u)$, that is $p^*(u-v)=0$. As the pull-back induced by a surjective map is injective, we conclude that $v=u$ that is $p_*(p^*(u)=u$.

\end{proof}

 \subsection{The groups of symmetric matrices}
 For each non-negative integer $n$, we let  $\mysms{n}{R}$ denote  the subgroup of all symmetric  matrices in $R^{n\times n}$ and define their direct sum 
 \begin{equation}
\mysm{R}:=\bigoplus_{n\in\Z_{\ge0}}\mysms{n}{R}.
\end{equation}
 Remark that $\mysms{0}{R}=\{[]\}$ where $\myem$
 is called  the \emph{\color{blue} empty matrix}. 
 \subsection{Canonical matrix bases} Explicit matrices  are canonically written in the form of two-dimensional tables of coefficients by using the canonical linear order in the set of integers and the canonical linear basis of $R^n$ given by the characteristic functions of the singletons $\chi_{\{i\}}$, $i\in n$. Interpreting the latter as column vectors, the products $E_{i,j}:=\chi_{\{i\}}\mytr{(\chi_{\{j\}})}$ for $(i,j)\in m\times n$, form the canonical linear basis of $R^{m\times n}$.
  \subsection{A pairing $4\times 4\to 3$} For any $(i,j)\in 4\times 4$, we define  a (symmetric) pairing $i\cdot j\in 3$ by the formula
 \begin{equation}
i\cdot j:=\mytr{v_i}v_j
\end{equation}
where 
\begin{equation}
v\colon 4\to \C^3, \quad v_{k+2l}=\chi_{\{2\}}+\left(\sqrt{-1}\right)^{k+(1-k)2l}\chi_{\{k\}},\quad k,l\in 2,
\end{equation}
or in the form of an explicit matrix
\begin{equation}
\left[i\cdot j\right]_{i,j\in 4}=\begin{bmatrix}
 2&1&0&1\\
 1&0&1&0\\
 0&1&2&1\\
 1&0&1&0
\end{bmatrix}.
\end{equation}
\subsection{The $SR$-equivalence relation}
 In the space of symmetric matrices, the $S$-equivalence  was introduced by  Trotter~\cite{MR0143201} to construct the signature and nullity invariants of knots, and by Murasugi~\cite{MR0171275} for their extensions to links, and further extensions to link signature functions by Tristram~\cite{MR0248854} and Levine~\cite{MR0253348}. 
 
 For our purposes, we define the \emph{\color{blue}$SR$-equivalence} relation in the space of pairs $(m, A)$, where $m$ is an element of the additive group $\Z^2$ and $A\in \mysm{R}$,  to be generated by the congruences $(m,A)\sim (m, \mytr{M}AM)$ for all invertible matrices $M$ of the same size  as that of $A$ (but not necessarily symmetric) and the modified $S$-equivalences $\left(m, 
\begin{bmatrix}
  0&1\\ 1&r
\end{bmatrix}\boxplus A
\right)\sim\left(m+1,A\right)$ for all $r\in R$, and $\left(m,[\epsilon]\boxplus A \right)\sim\left(m+\frac{1+\epsilon-2\epsilon\chi_{\{1\}}}{1+\epsilon^2},A\right)$ for all $\epsilon \in \{1,-1,0\}$. 
\begin{lemma}\label{mainlem}
For a ring $R$ and a positive integer $n$, let  $\zeta\in R$ be  an invertible element, $A\in \mysms{n}{R}$,  $v\in R^n$, and $r\in R$. Then, the matrix
 \begin{equation}
B=\begin{bmatrix}
0&\zeta&\begin{matrix}0&\cdots&0\end{matrix}\\
\zeta &r&\mytr{v}\\
\begin{matrix}0\\ \vdots\\ 0\end{matrix}&v& A
\end{bmatrix}\in \mysms{n+2}{R}
\end{equation}
is congruent to $\begin{bmatrix}
 0&\zeta\\\zeta&r
\end{bmatrix}\boxplus  A$. 
\end{lemma}
\begin{proof}
 We have
 \begin{equation}
(1-u\mytr{w})B(1-w\mytr{u})=
\begin{bmatrix}
 0&\zeta\\\zeta&r
\end{bmatrix}\boxplus  A
\end{equation}
where $u,w\in R^{n+2}$ are defined by 
\begin{equation}
u:=\zeta^{-1}\sum_{i=2}^{n+1}v_i\chi_{\{i\}},\quad w:=\chi_{\{0\}},
\end{equation}
and the matrix $1-u\mytr{w}$ is invertible with inverse $1+u\mytr{w}$ due to the orthogonality relation $\mytr{w}u=0$.
\end{proof}
\begin{lemma}\label{lemma-about-pf}
For two positive integers $m$ and $n$ and a map $f\colon m\to n$, let $w\in R^m$ be  a $f$-compatible vector. Then, for any $A\in \mysms{m}{R}$ and $v\in R^m$, one has
\begin{equation}
f_*\left(\left(1+v\mytr{w}\right)A\left(1+w\mytr{v}\right)\right)=\left(1+\tilde v\mytr{\tilde w}\right)f_*(A)\left(1+\tilde w\mytr{\tilde v}\right).
\end{equation}
where we denote $\tilde u:=f_*(u)$ for $u\in\{v,w\}$.
\begin{proof} By repeatedly using formula~\eqref{proppf}, we have
\begin{multline}
 f_*\left(\left(1+v\mytr{w}\right)A\left(1+w\mytr{v}\right)\right)=f_*\left(\left(A+v\mytr{w}A\right)\left(1+w\mytr{v}\right)\right)\\
=f_*\left(A\right)+f_*\left(v\mytr{w}A\right)+ f_*\left(Aw\mytr{v}\right)+f_*\left(v\mytr{w}Aw\mytr{v}\right)\\
=f_*\left(A\right)+\tilde v\mytr{\left(f_*\left(Aw\right)\right)}+ f_*\left(Aw\right)\mytr{\tilde v}+ \tilde v\mytr{w}Aw\mytr{ \tilde v}\\
=f_*\left(A\right)+\tilde v\mytr{\left(f_*\left(A\right)\tilde w\right)}+ f_*\left(A\right)\tilde w\mytr{\tilde v}+ \tilde v\mytr{\tilde w} f_*\left(A\right)\tilde w\mytr{ \tilde v}\\=\left(1+\tilde v\mytr{\tilde w}\right)f_*(A)\left(1+\tilde w\mytr{\tilde v}\right).
\end{multline}
\end{proof}

\end{lemma}
\section{Construction of a link invariant}\label{sec2}

\subsection{Amplitude of a link diagram}
Let $D$ be an oriented link diagram. We let $\myc{D}$ and $\myr{D}$ denote the sets of all crossings and regions of $D$, respectively. Any $c\in\myc{D}$ is also  identified with a map 
$
c\colon 4\to \myr{D}
$
enumerating the regions around $c$ according to Fig.~\ref{regs-ar-cr}.
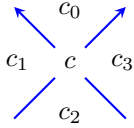
\begin{figure}[h]
\begin{center}
\begin{tikzpicture}[baseline,scale=.7]
\node (c) at (0,0) {$c$};
\node (nw) at (180:1){$c_1$};\node (ne) at (90:1){$c_0$};
\node (se) at (0:1){$c_3$};\node (sw) at (-90:1){$c_2$};

 \draw[ thick, >=stealth,->,color=blue](-135:1.5)--(c)--(45:1.5);
 \draw[ thick, >=stealth,->,color=blue](-45:1.5)--(c)--(135:1.5);
 
\end{tikzpicture}
\end{center}
\caption{The regions around a crossing enumerated by the elements of the ordinal number $4$.}\label{regs-ar-cr}
\end{figure}
The \emph{\color{blue}amplitude} of $D$
is a matrix $\mym{D}\in\Z[x]^{\myr{D}\times\myr{D}}$ defined by the formula
\begin{equation}
\mym{D}:=\sum_{c\in \myc{D}}\mym{c},\quad 
\mym{c}:=\operatorname{sgn}(c)\sum_{(i,j)\in 4\times 4}h_{i\cdot j}\chi_{\{(c_i,c_j)\}},
\end{equation}
where $\operatorname{sgn}(c)\in\{\pm1\}$ is the sign of the crossing $c$, and
\begin{equation}
h_0:=1,\quad h_1:=x,\quad h_2:=2x^2-1
\end{equation}
are the first three Hermite polynomials normalised by removing all integer overall multiplicative factors.
\begin{remark}
 It is easily seen that for any oriented link diagram $D$, its amplitude is a symmetric  matrix which finds itself  in the subgroup $\Z[2x]^{\myr{D}\times\myr{D}}\subset \Z[x]^{\myr{D}\times\myr{D}}$. Nonetheless, eventually it will be important for us to consider  $\mym{D}$ as an element of even larger group $\Z[x,\frac1{2x+1}]^{\myr{D}\times\myr{D}}$.
\end{remark}
\begin{remark}
 The ring automorphism of $\Z[x]$ trivially acting on the integers and changing the sign of  the indeterminate $x$ gives rise to amplitudes congruent to the original ones through the sign changing diagonal matrices with the negative entries associated with regions related to each other in the checkerboard manner. That means that in realisations one can always use some preferred choice between $\pm x$, for example, $x\ge 0$ if $x\in \R$.
\end{remark}

 \subsection{The main result}
\begin{theorem}\label{thm-sr-equiv}
 Let $D$ be an oriented link diagram with $\myc{+}$ (respectively $\myc{-}$) positive (respectively negative) crossings  and $R:=\Z[x,\frac1{2x+1}]$ (or any quotient thereof embedded into an arbitrary ring). Then, the $SR$-equivalence class of the pair $(-m_D,\mym{D})$, where $(m_{D})_{\frac{1\mp1}2}=\myc{\pm}$, is constant on the Reidemeister equivalence class of $D$.
\end{theorem}
\begin{proof} (RIII)  We start by analysing the Reidemeister moves involving only positive crossings. The cases with negative crossings are controlled by combining with moves of type II. Let us consider the graphs
 \begin{equation}\label{reiIII}
L_3:=\quad
\begin{tikzpicture}[baseline=41,scale=1.5]
 \coordinate (s0)  at (0,0);
 \coordinate (s1)  at (1,0);
 \coordinate (s2)  at (2,0);
 \coordinate (n0)  at (0,2);
 \coordinate (n1)  at (1,2);
 \coordinate (n3)  at (2,2);

\draw[->,>=stealth, 
thick,
color=blue] (s2) 
--(n0);
\draw[line width=5,color=white] (s1) to [out=150,in=-150] 
(n1);
\draw[->,>=stealth, 
thick,
color=blue](s1) to [out=150,in=-150] 
(n1);

\draw[line width=5,color=white] (s0)
--(n3);
\draw[->,>=stealth, 
thick,
color=blue] (s0) 
--(n3);
\node[left] at (.8,1) {$0$};
\node[right] at (0,1) {$1$};
\node[below] at (.5,2) {$2$};
\node[below] at (1.5,2) {$3$};
\node[left] at (2,1) {$4$};
\node[above] at (1.5,0) {$5$};
\node[above] at (.5,0) {$6$};
\end{tikzpicture},\qquad
R_3:=\quad
\begin{tikzpicture}[baseline=41,scale=1.5]
 \coordinate (s0)  at (0,0);
 \coordinate (s1)  at (1,0);
 \coordinate (s2)  at (2,0);
 \coordinate (n0)  at (0,2);
 \coordinate (n1)  at (1,2);
 \coordinate (n3)  at (2,2);

\draw[->,>=stealth, 
thick,
color=blue] (s2) 
--(n0);
\draw[line width=5,color=white] (s1) to [out=30,in=-30] 
(n1);
\draw[->,>=stealth, 
thick,
color=blue](s1) to [out=30,in=-30] 
(n1);

\draw[line width=5,color=white] (s0)
--(n3);
\draw[->,>=stealth, 
thick,
color=blue] (s0) 
--(n3);
\node[right] at (1.2,1) {$0$};
\node[right] at (0,1) {$1$};
\node[below] at (.5,2) {$2$};
\node[below] at (1.5,2) {$3$};
\node[left] at (2,1) {$4$};
\node[above] at (1.5,0) {$5$};
\node[above] at (.5,0) {$6$};
\end{tikzpicture}
\end{equation}
where the regions of $L_3$ and $R_3$ are in a natural bijection, and we enumerate them by elements of $7$ through a chosen bijection of $7$ with
$\myr{L_3}\simeq\myr{R_3}$. Initially, in thinking of $L_3$ and $R_3$ as parts of link diagrams, we suppose that all regions are pairwise distinct.
The contributions to the amplitudes of link diagrams are described by the explicit matrices
\begin{equation}
\mym{L_3}=
\begin{bmatrix}
 4x^2-1&2x&1&2x&1&2x&1\\
2x &2&x&1&0&1&x\\
 1&x&2x^2-1&x&0&0&0\\
 2x&1&x&2x^2&x&1&0\\
 1&0&0&x&1&x&0\\
2x&1&0&1&x& 2x^2&x\\
 1&x&0&0&0&x&2x^2-1
\end{bmatrix}
\end{equation}
and
\begin{equation}
\mym{R_3}=
\begin{bmatrix}
 4x^2-1&1&2x&1&2x&1&2x\\
1 &1&x&0&0&0&x\\
2 x&x&2x^2&x&1&0&1\\
 1&0&x&2x^2-1&x&0&0\\
 2x&0&1&x&2&x&1\\
1&0&0&0&x& 2x^2-1&x\\
 2x&x&1&0&1&x&2x^2
\end{bmatrix}.
\end{equation}
It is easily verified that these matrices are congruent in the localised ring $R$ through the equality
\begin{equation}\label{reiIII-equiv}
\mym{L_3}=\left(1+v \mytr{w}\right)\mym{R_3}\left(1+w \mytr{v}\right)
\end{equation}
where the column vectors $v,w\in R^{7}$ are of the form
\begin{equation}\label{defs-of v-and-w}
v:=(2x+1)^{-1}\sum_{i=1}^6(-1)^{i-1}\chi_{\{i\}},\quad w:=\chi_{\{0\}}
\end{equation}
thus satisfying the orthogonality relation 
\begin{equation}\label{bigMij}
\mytr{w} v=0
\end{equation}
 which implies that $(1+v \mytr{w})^{-1}=1-v \mytr{w}$. That means that  $\mym{L_3}$ and $\mym{R_3}$ are congruent and so are the amplitudes of any pair of link diagrams $D$ and $D'$ differing only in the parts given by $L_3$ and $R_3$ and when all concerned regions are pairwise distinct. Indeed, if $n\ge 7$ is the number of regions in $D$ and $D'$, then
 \begin{equation}
\mym{D}=\left(\mym{L_3}\boxplus [0]^{\boxplus(n-7)}\right)+\left([0]\boxplus A\right),\quad \mym{D'}=\left(\mym{R_3}\boxplus [0]^{\boxplus (n-7)}\right)+\left([0]\boxplus A\right),
\end{equation}
where $A\in\mysms{n-1}{R}$ is the contribution from the common parts of $D$ and $D'$, and, due to equality~\eqref{reiIII-equiv}, we obtain
\begin{equation}
\mym{D}=\left(1+f_*(v) \mytr{f_*(w)}\right)\mym{D'}\left(1+f_*(w) \mytr{f_*(v)}\right)
\end{equation}
where $f\colon 7\to n$ is the canonical inclusion. As any vector is $f$-compatible, the push-forwards of $v$ and $w$ are still orthogonal so that the matrix $1+f_*(v) \mytr{f_*(w)}$ is invertible. Thus, $\mym{D}$ and $\mym{D'}$ are congruent over the localised ring $R$.

In order to verify the equivalence under all Reidemeister moves of type III involving only positive crossings, we also need to verify the equivalence of the push-forwards of the amplitudes of two link diagrams differing from each other in the parts $L_3$ and $R_3$, where some of the regions in \eqref{reiIII} are identified. Lemma~\ref{lemma-about-pf} allows to push forward equality~\eqref{reiIII-equiv} by an identification map $f$, provided $w$ is $f$-compatible. By Lemma~\ref{3-equiv-cond} and the form of $w$ in \eqref{defs-of v-and-w}, that is the case as soon as $|f^{-1}(f_0)|=1$, and, due to the orthogonality of push-forwards of $v$ and $w$, we conclude that the push-forwards of $\mym{L_3}$ and $\mym{R_3}$ are still congruent, and so are the amplitudes of diagrams containing them and only differing in them.
 That proves that all Reidemeister moves of type III involving only positive crossings give rise to congruent amplitudes.

(RII:1) For the first instance of the Reidemeister move of type II, we consider the graph
\begin{equation}\label{reiII}
L_{2.1}:=
\begin{tikzpicture}[baseline=7.3,scale=1.5,spin/.style={circle,draw=black!50,fill=black!10,thick}]
\coordinate (ne)  at (2,.5);
\coordinate (nw)  at (0,.5);
\coordinate (se)  at (2,0);
\coordinate (sw)  at (0,0);
\node (w) at (0,.25){$1$};\node (n) at (1,.6){$4$};
\node (e) at (2,.25){$2$};\node (s) at (1,-.1){$3$};
\node (c) at (1,.25){$0$};
\draw[ thick, >=stealth,->,color=blue](se) to [out=135,in=45](sw);
\draw[line width=5,color=white](ne) to [out=-135,in=-45](nw);
\draw[thick, >=stealth,->,color=blue](ne) to [out=-135,in=-45](nw);
\end{tikzpicture}
\end{equation}
whose amplitude reads
\begin{equation}
\mym{L_{2.1}}=
\begin{bmatrix}
 0&1&-1&0&0\\
1  &2x^2-1&0&x&x\\ 
 -1 &0&1-2x^2&-x&-x\\ 
  0&x&-x&0&0\\ 
 0 &x&-x&0&0
\end{bmatrix}.
\end{equation}

Any link diagram $D$ with $n\ge5$ regions, containing $L_{2.1}$  with the pairwise distinct regions, has the amplitude of the form
\begin{equation}
\mym{D}=(\mym{L_{2.1}}\boxplus  [0]^{n-5})+([0]\boxplus  A)
\end{equation}
with some $A\in\mysms{n-1}{R}$. Now, the equality
\begin{equation}\label{eq-with-e12}
(1+E_{2,1})\mym{L_{2.1}}(1+E_{1,2})=
\begin{bmatrix}
 0&1&0&0&0\\
1  &2x^2-1&2x^2-1&x&x\\ 
0&2x^2-1&0&0&0\\
0&x&0&0&0\\
0&x&0&0&0
\end{bmatrix}
\end{equation}
and Lemma~\ref{mainlem} imply that $\mym{D}$ is congruent to $\begin{bmatrix}
 0&1\\1&r
\end{bmatrix}\boxplus  f_*(A)$
with some $r\in R$ and where
 \begin{equation}
f\colon n-1\to n-2,\quad i\mapsto i-1+\delta_{0,i},
\end{equation}
is the map which identifies the regions $1$ and $2$ in the diagram $D'$ which is $D$ where the part $L_{2.1}$ is replaced by 
\begin{equation}
R_{2.1}:=
\begin{tikzpicture}[baseline=7.3,scale=1.5,spin/.style={circle,draw=black!50,fill=black!10,thick}]
\coordinate (ne)  at (.5,.5);
\coordinate (nw)  at (0,.5);
\coordinate (se)  at (.5,0);
\coordinate (sw)  at (0,0);
\node (w) at (0.25,.25){$2$};
\node (n) at (.25,.6){$4$};
\node (c) at (0.25,-.1){$3$};
\draw[ thick, >=stealth,->,color=blue](se) to [out=135,in=45](sw);
\draw[thick, >=stealth,->,color=blue](ne) to [out=-135,in=-45](nw);

\end{tikzpicture}\ ,
\end{equation}
so that $f_*(A)=\mym{D'}$, and the role of the push-forward map $f_*$ is to implement the identification of the regions $1$ and $2$ in $\mym{D'}$. Taking into account the differences in the numbers of positive and negative crossings in $L_{2.1}$ and $R_{2.1}$, we have the equality $m_D=1+m_{D'}$, and  thus we obtain the $SR$-equivalence
\begin{equation}\label{SR-eq-for-r211}
(-m_D,\mym{D})\sim \left(-m_{D},\begin{bmatrix}
 0&1\\
1  &r
\end{bmatrix}\boxplus  \mym{D'}\right)\sim \left(1-m_{D},\ \mym{D'}\right)= (-m_{D'},\mym{D'}).
\end{equation}
This equivalence projects down to any further identification of the outer regions of $L_{2.1}$ provided the regions $1$ and $2$ are not identified in $D$. The reason for the latter restriction comes from the matrix $E_{2,1}=v\mytr{w}$ in the congruence transformation~\eqref{eq-with-e12} which is such that, by Lemma~\ref{3-equiv-cond}, the vector $w=\chi_{\{1\}}$ is not compatible with a projection which identifies the regions $1$ and $2$, and thus, Lemma~\ref{lemma-about-pf} cannot be used for pushing forward this congruence relation. This is why we have to analyse this case separately.

(RII:2) We consider the graph
\begin{equation}\label{reiII-212}
L_{2.2}:=
\begin{tikzpicture}[baseline=7.3,scale=1.5,spin/.style={circle,draw=black!50,fill=black!10,thick}]
\coordinate (ne)  at (2,.5);
\coordinate (nw)  at (0,.5);
\coordinate (se)  at (2,0);
\coordinate (sw)  at (0,0);
\node (w) at (0,.25){$1$};\node (n) at (1,.6){$3$};
\node (e) at (2,.25){$1$};\node (s) at (1,-.1){$2$};
\node (c) at (1,.25){$0$};
\draw[ thick, >=stealth,->,color=blue](se) to [out=135,in=45](sw);
\draw[line width=5,color=white](ne) to [out=-135,in=-45](nw);
\draw[thick, >=stealth,->,color=blue](ne) to [out=-135,in=-45](nw);
\end{tikzpicture}
\end{equation}
with the associated amplitude
\begin{equation}
\mym{L_{2.2}}=[0]^{\boxplus  4}.
\end{equation} 
Any link diagram $D$ with $n\ge4$ regions, containing $L_{2.2}$, has the amplitude of the form
\begin{equation}
\mym{D}=(\mym{L_{2.2}}\boxplus  [0]^{n-4})+([0]\boxplus  A)=[0]\boxplus  A
\end{equation}
with some $A\in\mysms{n-1}{R}$. On the other hand, $A$ is the amplitude of the diagram $D'$ obtained from $D$ by replacing $L_{2.2}$ by the diagram
\begin{equation}
R_{2.2}:=
\begin{tikzpicture}[baseline=7.3,scale=1.5,spin/.style={circle,draw=black!50,fill=black!10,thick}]
\coordinate (ne)  at (.5,.5);
\coordinate (nw)  at (0,.5);
\coordinate (se)  at (.5,0);
\coordinate (sw)  at (0,0);
\node (w) at (.25,.25){$1$};\node (n) at (.25,.6){$3$};
\node (c) at (0.25,-.1){$2$};
\draw[ thick, >=stealth,->,color=blue](se) to [out=135,in=45](sw);
\draw[thick, >=stealth,->,color=blue](ne) to [out=-135,in=-45](nw);

\end{tikzpicture}\ .
\end{equation}
By using the same relation $m_D=1+m_{D'}$ as in the previous case, we obtain another $SR$-equivalence
\begin{equation}
(-m_{D},\mym{D})=\left(-m_{D},[0]\boxplus  \mym{D'}\right)\sim \left(1-m_{D}, \mym{D'}\right)=(-m_{D'},\mym{D'}).
\end{equation}
(RII:3) Next, we consider the graph
\begin{equation}\label{reiII-221}
L_{2.3}:=
\begin{tikzpicture}[baseline=7.3,scale=1.5,spin/.style={circle,draw=black!50,fill=black!10,thick}]
\coordinate (ne)  at (2,.5);
\coordinate (nw)  at (0,.5);
\coordinate (se)  at (2,0);
\coordinate (sw)  at (0,0);
\node (w) at (0,.25){$1$};\node (n) at (1,.6){$4$};
\node (e) at (2,.25){$2$};\node (s) at (1,-.1){$3$};
\node (c) at (1,.25){$0$};
\draw[ thick, >=stealth,->,color=blue](se) to [out=135,in=45](sw);
\draw[line width=5,color=white](ne) to [out=-135,in=-45](nw);
\draw[thick, >=stealth,->,color=blue](nw) to [out=-45,in=-135](ne);
\end{tikzpicture}
\end{equation}
whose amplitude reads
\begin{equation}
\mym{L_{2.3}}=
\begin{bmatrix}
 0&-1&1&0&0\\
-1  &-1&0&-x&-x\\ 
 1 &0&1&x&x\\ 
  0&-x&x&0&0\\ 
 0 &-x&x&0&0
\end{bmatrix}.
\end{equation}

Any link diagram $D$ with $n\ge5$ regions, containing $L_{2.3}$  with the pairwise distinct regions, has the amplitude of the form
\begin{equation}
\mym{D}=(\mym{L_{2.3}}\boxplus  [0]^{n-5})+([0]\boxplus  A)
\end{equation}
with some $A\in\mysms{n-1}{R}$. Now, the equality
\begin{equation}\label{eq-with-e12-r221}
(1+E_{2,1})\mym{L_{2.3}}(1+E_{1,2})=
\begin{bmatrix}
 0&-1&0&0&0\\
-1  &-1&-1&-x&-x\\ 
0&-1&0&0&0\\
0&-x&0&0&0\\
0&-x&0&0&0
\end{bmatrix}
\end{equation}
and Lemma~\ref{mainlem} imply that $\mym{D}$ is congruent to $\begin{bmatrix}
 0&-1\\-1&r
\end{bmatrix}\boxplus  f_*(A)$
with some $r\in R$ and where
 \begin{equation}
f\colon n-1\to n-2,\quad i\mapsto i-1+\delta_{0,i},
\end{equation}
is the map which identifies the regions $1$ and $2$ in the diagram $D'$ which is $D$ where the part $L_{2.3}$ is replaced by 
\begin{equation}
R_{2.3}:=
\begin{tikzpicture}[baseline=7.3,scale=1.5,spin/.style={circle,draw=black!50,fill=black!10,thick}]
\coordinate (ne)  at (.5,.5);
\coordinate (nw)  at (0,.5);
\coordinate (se)  at (.5,0);
\coordinate (sw)  at (0,0);
\node (w) at (0.25,.25){$2$};
\node (n) at (.25,.6){$4$};
\node (c) at (0.25,-.1){$3$};
\draw[ thick, >=stealth,->,color=blue](se) to [out=135,in=45](sw);
\draw[thick, >=stealth,->,color=blue](nw) to [out=-45,in=-135](ne);

\end{tikzpicture}\ ,
\end{equation}
so that $f_*(A)=\mym{D'}$, and the role of the push-forward map $f_*$ is to implement the identification of the regions $1$ and $2$ in the amplitude of the diagram $D'$. Furthermore, the matrix identity
\begin{equation}
\begin{bmatrix}
 1&0\\
 0&-1
\end{bmatrix}
\begin{bmatrix}
 0&-1\\
 -1&r
\end{bmatrix}
\begin{bmatrix}
 1&0\\
 0&-1
\end{bmatrix}
=\begin{bmatrix}
 0&1\\
 1&r
\end{bmatrix}
\end{equation}
ensures that the matrices $\begin{bmatrix}
 0&-1\\
 -1&r
\end{bmatrix}$ and $\begin{bmatrix}
 0&1\\
 1&r
\end{bmatrix}$
are congruent.  Thus, again taking into account the differences in the numbers of positive and negative crossings in $L_{2.3}$ and $R_{2.3}$, we have the equality $m_D=1+m_{D'}$, and we obtain an $SR$-equivalence of the form~\eqref{SR-eq-for-r211}. Similarly to the case with diagram~\eqref{reiII}, we can push forward the congruence~\eqref{eq-with-e12-r221} to the cases with outer region identifications under the condition that the regions $1$ and $2$ stay distinct in $D$. The case, where the latter condition is not satisfied, is handled separately in complete analogy with diagram~\eqref{reiII-212}.

(RI) For the Reidemeister moves of type I, we consider a diagram with positive crossing
\begin{equation}\label{rI+}
L_1:=\begin{tikzpicture}[baseline=-2,scale=.5,spin/.style={circle,draw=black!50,fill=black!10,thick}]
\coordinate (ne)  at (1,.5);
\coordinate (nw)  at (-1,1);
\coordinate (se)  at (1,-.5);
\coordinate (sw)  at (-1,-1);
\node[left] (w) at (-.5,0){$1$};\node[right] (n) at (1.5,0){$2$};\node (e) at (.6,0){$0$};

\draw[ thick, >=stealth,->,color=blue](se) to [out=-135,in=-45](nw);
\draw[line width=5,color=white](sw) to [out=45,in=135](ne);
\draw[thick, color=blue](sw) to [out=45,in=135](ne);
\draw[thick, color=blue](ne) to [out=-45,in=45](se);
\end{tikzpicture}
\end{equation}
with the associated amplitude
\begin{equation}
\mym{L_1}=
\begin{bmatrix}
 1&1&2x\\
 1&1&2x\\
 2x&2x&4x^2
\end{bmatrix}
\end{equation}
which satisfies the congruence relation
\begin{equation}\label{cong-for-r1}
(1-v\mytr{w})\mym{L_1}(1-w\mytr{v})=[1]\boxplus  [0]^{\boxplus  2}.
\end{equation}
where
\begin{equation}
v:=\chi_{\{1\}}+2x\chi_{\{2\}},\quad w:=\chi_{\{0\}}.
\end{equation}
Let $D$  be a link diagram with $n\ge3$ crossings containing $L_1$ as a part. The amplitude of $D$ is given by the formula
\begin{equation}
\mym{D}=\left(\mym{L_1}\boxplus [0]^{\boxplus(n-3)}\right)+\left([0]\boxplus A\right)
\end{equation}
where $A\in \mysms{n-1}{R}$. By the canonical embedding $f\colon 3\to n$, the congruence relation~\eqref{cong-for-r1} in $R^3$ is pushed forward to the congruence relation in $R^n$
\begin{equation}
(1-f_*(v)\mytr{\left(f_*(w)\right)})\mym{L_1}(1-f_*(w)\mytr{\left(f_*(v)\right)})=[1]\boxplus \mym{D'}
\end{equation}
where $\mym{D'}=A$ is the amplutude of the diagram $D'$  obtained from $D$ by replacing the part $L_1$ by the diagram
\begin{equation}
R_1:=
\begin{tikzpicture}[baseline=-2,scale=.5,spin/.style={circle,draw=black!50,fill=black!10,thick}]
\coordinate (nw)  at (-1,1);
\coordinate (sw)  at (-1,-1);
\node[left] (w) at (-1,0){$1$};
\node[right] (e) at (-.2,0){$2$};

\draw[ thick, >=stealth,->,color=blue](sw) to [out=45,in=-45](nw);
\end{tikzpicture}\ .
\end{equation}
Taking into account the equality $m_D=\chi_{\{0\}}+m_{D'}$, we 
 obtain  an $SR$-equivalence
\begin{equation}
(-m_{D},\mym{D})\sim (-m_{D},[1]\boxplus \mym{D'})\sim (\chi_{\{0\}}-m_{D},\mym{D'})\sim(-m_{D'},\mym{D'}).
\end{equation}
All other cases, including those with a negative crossing in diagram~\eqref{rI+} is analysed in a completely similar way.
\end{proof}
\section{Examples of results of calculation}\label{sec3}
\subsection{The $n$ component trivial link}
By considering the simplest diagram
\begin{equation}
D=
\begin{tikzpicture}[baseline=-2,scale=1]
\draw[thick, directed,color=blue] (0,0) circle (.5);
\draw[thick, directed,color=blue] (1.5,0) circle (.5);
\draw[thick, directed,color=blue] (3.5,0) circle (.5);
\node at (0,0) {$0$};\node at (1.5,0) {$1$};
\node at (2.5,0) {$\ldots$};
\node at (3.5,0) {${n-1}$};
\node at (4.5,0) {${n}$};
\end{tikzpicture}
\end{equation}
with $(\myc{+},\myc{-})=(0,0)$, the amplitude is the trivial  zero matrix
$\mym{D}=[0]^{\boxplus  (n+1)}$
so that 
\begin{equation}
(-m_D,\mym{D})\sim (n+1,[]).
\end{equation}

\subsection{The right handed Hopf link} We consider the diagram
\begin{equation}
D=
\begin{tikzpicture}[baseline=-2,scale=1]
 \coordinate (e0)  at (0,0);
 \coordinate (e1)  at (1,0);
 \coordinate (e2)  at (2,0);
\coordinate (e3)  at (3,0);

\draw[thick, directed,color=blue] (e2) to [out=90,in=90] (e0);
\draw[thick, directed,color=blue] (e3) to [out=-90,in=-90] (e1);

\draw[line width=5,color=white] (e1) to [out=90,in=90] (e3);
\draw[thick,directed,color=blue] (e1) to [out=90,in=90] (e3);
\draw[line width=5,color=white](e0) to [out=-90,in=-90](e2);
\draw[thick,  directed,color=blue] (e0) to [out=-90,in=-90](e2);

\node at (.5,0) {$0$};
\node at (1.5,0) {$1$};
\node at (2.5,0) {$2$};
\node at (3.5,0) {$3$};
\end{tikzpicture}
\end{equation}
with $(\myc{+},\myc{-})=(2,0)$ so that the amplitude takes the form
\begin{equation}
\mym{D}=
\begin{bmatrix}
 2&y&2&y\\
y &y^2-2&y&2\\
2 &y&2&y\\
y &2&y&y^2-2
\end{bmatrix},\quad y:=2x,
\end{equation}
so that
\begin{equation}
(-m_D,\mym{D})\sim \left(2\chi_{\{1\}}, [2x^2-2]\boxplus  [2]\right).
\end{equation}
\subsection{The right handed trefoil}
We consider the diagram
\begin{equation}
D=
\begin{tikzpicture}[baseline=-2,scale=.7]
 \coordinate (e1)  at (0:1);
 \coordinate (e2)  at (0:2);
 \coordinate (nw1)  at (120:1);
\coordinate (nw2)  at (120:2);
\coordinate (sw1)  at (-120:1);
\coordinate (sw2)  at (-120:2);

\draw[thick, directed,color=blue] (e2) to [out=90,in=30] (nw1);
\draw[thick, directed,color=blue] (nw2) to [out=-150,in=150] (sw1);
\draw[thick,  directed,color=blue] (sw2) to [out=-30,in=-90] (e1);

\draw[line width=5,color=white] (e1) to [out=90,in=30] (nw2);
\draw[thick,directed,color=blue] (e1) to [out=90,in=30] (nw2);

\draw[line width=5,color=white] (nw1) to [out=-150,in=150] (sw2);
\draw[thick,  directed,color=blue] (nw1) to [out=-150,in=150] (sw2);

\draw[line width=5,color=white] (sw1) to [out=-30,in=-90] (e2);
\draw[thick, directed,color=blue] (sw1) to [out=-30,in=-90] (e2);

\node at (0,0) {$0$};
\node at (0:1.5) {$1$};\node at (120:1.5) {$2$};\node at (-120:1.5) {$3$};
\node at (0:3) {$4$};
\end{tikzpicture}
\end{equation}
with $(\myc{+},\myc{-})=(3,0)$. The amplitude takes the form
\begin{equation}
\mym{D}=
\begin{bmatrix}
 3&y &y &y &3\\
 y &y ^2-2&1&1&y \\
  y &1&y ^2-2&1&y \\
y &1&1&y ^2-2&y \\
3 &y &y &y &3
\end{bmatrix},\quad y:=2x,
\end{equation}
so that
\begin{equation}
(-m_D,\mym{D})\sim \left(2\chi_{\{1\}}, (y^2-3)[1]^{\boxplus  2}\right).
\end{equation}
\subsection{The figure-eight knot}
We take the diagram
\begin{equation}
D=
\begin{tikzpicture}[baseline=40,xscale=1.2,yscale=.8]
 \coordinate (c0)  at (0,0);
 \coordinate (c1)  at (0,1);

\coordinate (l1)  at (-.5,2);
\coordinate (r1)  at (.5,2);
\coordinate (r2)  at (.5,2.8);
\coordinate (l2)  at (-.5,2.8);
\coordinate (r3)  at (1.75,3);
\coordinate (l3)  at (-1.75,3);

\draw[thick, directed,color=blue] (c0) to [out=0,in=-45] (r1);
\draw[thick, directed,color=blue] (l2) to [out=90,in=135] (r3);
\draw[thick, color=blue] (c1) to [out=180,in=-135] (l3);
\draw[thick,color=blue] (r2) to [out=-90,in=45] (l1);

\draw[line width=5,color=white] (l1) to [out=-135,in=180] (c0);
\draw[thick,directed,color=blue](l1) to [out=-135,in=180] (c0);

\draw[line width=5,color=white] (r1) to [out=135,in=-90] (l2);
\draw[thick,directed,color=blue] (r1) to [out=135,in=-90] (l2);
\draw[line width=5,color=white] (r3) to [out=-45,in=0] (c1);
\draw[thick,directed,color=blue](r3) to [out=-45,in=0] (c1);
\draw[line width=5,color=white] (l3) to [out=45,in=90] (r2);
\draw[thick,directed,color=blue] (l3) to [out=45,in=90] (r2);

\node at (-1.5,.7) {$0$};
\node at (0,.5) {$1$};
\node at (0,1.5) {$2$};
\node at (0,2.8) {$3$};
\node at (-1,2.5) {$4$};
\node at (1,2.5) {$5$};

\end{tikzpicture}
\end{equation}
with $(\myc{+},\myc{-})=(2,2)$. The amplitude takes the form
\begin{equation}
\mym{D}=
\begin{bmatrix}
3-y^2 &-y&-2&1&0&0\\
-y &-2&-y&0&-1&-1\\
-2 &-y&3-y^2 &1&0&0\\
 1&0&1&2&y&y\\
0 &-1&0&y&y^2-3&2\\
0 &-1&0&y&2&y^2-3
\end{bmatrix},\quad y:=2x,
\end{equation}
so that
\begin{equation}
(-m_D,\mym{D})\sim \left(1, (y^2-5)
\begin{bmatrix}
 0&1\\1&0
\end{bmatrix}
\right).
\end{equation}
\def\cprime{$'$} \def\cprime{$'$}

  \end{document}